\documentclass[11pt,oneside]{article}
\usepackage{color}
\usepackage{xcolor}
\usepackage{psfrag}
\usepackage{epsfig}
\usepackage{graphicx}
\usepackage{amsmath}
\usepackage{amssymb}
\usepackage{amsthm}
\usepackage{cite}
\usepackage[latin1]{inputenc}

\newtheorem{theorem}{Theorem}[section]
\newtheorem{definition}[theorem]{Definition}

\newtheorem{proposition}[theorem]{Proposition}

\newtheorem{corollary}[theorem]{Corollary}

\newcommand{\numberset}{\mathbb}

\newcommand{\F}{\numberset{F}}

\newenvironment{sistema}
{\left\lbrace\begin{array}{@{}l@{}}}
	{\end{array}\right.}

\begin{document}

\title{Minimum weight codewords in dual Algebraic-Geometric codes from the Giulietti-Korchm\'aros curve}
\date{}
\author{Daniele Bartoli, Matteo Bonini}

\maketitle

\begin{abstract}
In this paper we investigate the number of minimum weight codewords of some dual Algebraic-Geometric codes associated with the Giulietti-Korchm\'aros maximal curve, by computing the maximal number of intersections between the Giulietti-Korchm\'aros curve and lines, plane conics and plane cubics.
\end{abstract}


\section{Introduction}
%In this work we analyze minimal weight codewords of some one point AG codes from the GK curve.
%\begin{enumerate}
%\item Aggiungere tutta la bibliografia.
%\item Scrivere per bene cosa si sa citando i lavori.
%\item un po' di storia sullo stesso problema relativo all'hermitiana.
%\end{enumerate}

%\begin{enumerate}
%\item In generale nella GK le rette hanno al pi\`u $\ell^2-\ell+1$ punti di intersezione. Inoltre se la secante contiene un punto $\mathbb{F}_{\ell^2}$-razionale la retta ha solo al pi\`u $\ell+1$ punti. Se invece ha solo punti non $\mathbb{F}_{\ell^2}$-razionali allora se ha pi\`u di $\ell+1$ punti questi appartengono alla retta $x=\overline x, y =\overline y$.
%\item Quindi le $\ell^2-\ell+1$-secanti sono in tutto $(\ell+1)(\ell^5-\ell^3)$ (semplicemente contando il numero di punti della GK non $\mathbb{F}_{\ell^2}$-razionali e osservando che $z^{\ell^2-\ell+1}$ ha $0,1,\ell^2-\ell+1$ soluzioni. 
%\item Lo spazio di Riemann-Roch $\mathcal{L}(m(\ell^3+1)P_{\infty})$ \`e generato da $x^i,y^j,z^k$ dove $i\in [0,\ldots,\ell-1]$, $j \in [0,\ldots, \ell^2-\ell]$, $k\in [0,\ldots, m]$. Questo ci dice che applicando lo stesso ragionamento di \cite[Theorem 5]{MPS2016}, sapendo che le parole codice corrispondono a punti allineati, possiamo contare le parole codice. 
%\item Chiaramente stiamo utilizzando che $m\geq \ell$ per avere che i punti allineati devono essere per forza su $\ell^2-\ell+1$-secanti.
%\item Discutere al variare di $m$ quando la designed minimum distance \`e negativa e quindi quando sappiamo dire qualcosa di meglio. 

%\end{enumerate}

Algebraic-Geometric codes (AG codes for short) are an important class of error correcting codes; see \cite{Goppa81,Goppa82,S2009}. 

Let $\mathcal{X}$ be an algebraic curve defined over the finite field $\mathbb{F}_q$ of order $q$. The parameters of the AG codes associated with $\mathcal{X}$ strictly depend on some characteristics of the underlying curve $\mathcal{X}$. In general, curves with many $\mathbb{F}_q$-rational places with respect to their genus give rise to AG codes with good parameters. For this reason maximal curves, that is curves attaining the Hasse-Weil upper bound, have been widely investigated in the literature: for example the Hermitian curve and its quotients, the Suzuki curve, and the Klein quartic; see  \cite{Hansen1987,Matthews2004,Matthews2005,Stichtenoth1988,Tiersma1987,XC2002,XL2000,YK1992,BMZ2017_2}. More recently, AG codes were obtained from the Giulietti-Korchm\'aros curve \cite{GK2009} (GK curve for short), which is the first example of maximal curve shown not to be covered by the Hermitian curve; see \cite{FG2010,CT2016,BMZ2017}.

%The parameters of such codes strictly depend on the underlying curve $\mathcal{X}$. Maximal curves, that is curves having the maximum number of $\mathbb{F}_q$-rational places with respect to their genus given by the Hasse-Weil bound, give rise  to AG codes with good parameters. For this reason $\F_q$-maximal curves are widely studied in the literature: the best known example of maximal curves are the Hermitian curve and its quotients. Recently it was showed that the GK curve is the first example of a maximal curve not covered by the Hermitian curve and this is the motivation since AG codes from this curve are widely studied; see \cite{GK2009,BMZ2017,CT2016,D2011,FG2010}. 

In general, to know  the weight distribution of a particular code is a hard problem. Even the problem of computing  codewords of minimum weight can be a difficult  task apart from specific cases. In \cite{MPS2016}, following the approach of \cite{Sala,Augot}, the authors compute the number of minimum weight codewords of some dual AG codes from the Hermitian curve by providing an algebraic and geometric description for codewords of a given weight belonging to any fixed affine-variety code.

In this work we deal with AG codes arising from the GK maximal curve. The link between the minimum distance of such  codes and the underlying curve is given by a result of \cite{Couv2012}; see Theorem \ref{Th:Couvreur}.

In Section 2 we introduce basic notions and preliminary results concerning the GK curve, AG codes, and  affine variety codes. In Section 3 we compute the maximal intersections between the GK curve and lines, plane conics, and plane cubics. Such information is used in Section 4 to compute the number of  minimum weight codewords of some dual codes from the GK curve.

\section{Preliminary Results}

\subsection{The Giulietti-Korchm\'aros curve}
Let $\ell=p^h$, $p$ prime and $h\geq 1$. Denote by $PG(3,\ell^6)$ the three dimensional projective space over the field $\mathbb{F}_{\ell^6}$ with $\ell^6$ element. The Giulietti-Korchm\'aros curve $\mathcal{GK}$ (GK curve for short) is a non-singular curve in $PG(3,\ell^6)$, introduced in \cite{GK2009}, defined by the affine equations
\begin{equation}
\label{Eq:GK}
\begin{sistema}
Z^{\ell^2-\ell+1}=Y^{\ell^2}-Y\\
Y^{\ell+1}=X^\ell+X
\end{sistema}.
\end{equation}
This curve has genus $g=\frac{(\ell^3+1)(\ell^2-2)}{2}+1$, $\ell^8-\ell^6+\ell^5+1$ $\F_{\ell^6}$-rationals points and a unique point at the infinity $P_\infty$. The set $\mathcal{GK}(\F_{\ell^6})$ of the $\F_{\ell^6}$-rational points splits into two orbits under the action of $Aut(\mathcal{GK})$: the first one is composed by the set $\mathcal{GK}(\F_{\ell^2})$ of the $\F_{\ell^2}$-rational points of $\mathcal{GK}$, coinciding with the intersection between $\mathcal{GK}$ and the plane $Z=0$; the second one is formed by all the points in $\mathcal{GK}(\F_{\ell^6})\setminus\mathcal{GK}(\F_{\ell^2})$. The GK curve is $\F_{\ell^6}$-maximal, that is, it attains the Hasse-Weil bound $|\mathcal{GK}(\F_{\ell^6})|=\ell^6+1+2g\ell^3$; see  \cite[Theorem 5.2.3]{S2009}. Moreover, for $\ell>2$, $\mathcal{GK}$ is not covered by the Hermitian curve (see \cite{GK2009}): this is the first example in the literature of a family of maximal curves with this feature.

An algebraic curve $\mathcal{X}$ contained in a projective space of dimension $n$ is said to be a complete intersection if the ideal $I$ associated with $\mathcal{X}$ is generated by exactly $n-1$ polynomials. The curve $\mathcal{GK}$ is an example of a complete intersection curve in $PG(3,\ell^6)$.

Consider now the function field $\F_{\ell^6}(\mathcal{GK})$ associated with  $\mathcal{GK}$ (see \cite{S2009} for the connection between  function fields and curves) and let $x,y,z\in\F_{\ell^6}(\mathcal{GK})$ be its coordinate functions, which satisfy $y^{\ell+1}=x^{\ell}+x$ and $z^{\ell^2-\ell+1}=y^{\ell^2}-y$. 

A divisor $D$ on $\mathcal{GK}$ is a formal sum $n_1P_1+\cdots+n_rP_r$, where $P_i \in \mathcal{GK}(\mathbb{F}_{q})$, $n_i \in \mathbb{Z}$, $P_i\neq P_j$ if $i\neq j$.
The divisor $D$ is $\mathbb F_q$-rational if it coincides with its image $n_1P_1^q+\cdots+n_rP_r^q$ under the Frobenius map over $\mathbb F_q$.
For a function $f \in \mathbb{F}_q(\mathcal{GK})$, $div(f)$ indicates the divisor of $f$. 

Concerning the functions $x,y,z\in\F_{\ell^6}(\mathcal{GK})$ it is easily proved that
\begin{itemize}
	\item $(x)=(\ell^3+1)P_0-(\ell^3+1)P_{\infty}$,
	\item $(y)=(\ell^2-\ell+1)(\sum_{a: a^{\ell}+a=0}P_{(a,0,0)})-(\ell^3-\ell^2+\ell)P_{\infty}$,
	\item $(z)=\left(\sum_{P\in\mathcal{X}(\F_{\ell^2})\setminus\{P_{\infty}\}}P\right) -\ell P_{\infty}$,
\end{itemize}
where $P_{(a,b,c)}$ denotes the affine point $(a,b,c)$ and $P_0=P_{(0,0,0)}$.

\subsection{Algebraic-Geometric codes}

In this section we introduce some basics notions on AG codes. For a detailed introduction we refer to \cite{S2009}.

Let $\mathcal{X}$ be a projective curve over the finite field $\F_q$, consider the rational function field $\F_q(\mathcal{X})$ and the set $\mathcal{X}(\F_q)$ given by the $F_q$-rational places of $\mathcal{X}$.
Given a $\F_q$-rational divisor $D=\sum_{i=1,\dots,n} m_iP_i$ the Riemann-Roch space associated to $D$ on $\mathcal{X}$ is the vector space $\mathcal{L}(D)$ over $\F_q(\mathcal{X})$ is defined as
\[
\mathcal{L}(D)=\{f\in\F_{q}(\mathcal{X}) \ |\ (f)+D\ge0\}\cup \{0\}.
\]
It is known that this is a finite dimensional $\F_q$-vector space and the exact dimension can be computed using Riemann-Roch theorem.

Consider now the divisor $D=\sum_{P_i\ne P_j}P_i$ where all the $P_i$s have weight one. Let $G$ be another $F_q$-rational divisor such that $supp(G)\cap supp(D)=\emptyset$. Consider the evaluation map 
\[
\begin{split}
e_D:&\mathcal{L}(G)\rightarrow \F_q^n\\
&f\quad \longmapsto e_D(f)=(f(P_1,\dots,f(P_n)).
\end{split}
\]
This map is $\F_q$-linear and it is injective if $n>\deg(G)$.

The AG-code $C_{\mathcal{L}}(D,G)$ associated with the divisors $D$ and $G$ is then defined as $e_D(\mathcal{L}(G))$. It is an $[n,\ell(G)-\ell(G-D),d]_q$ code, where $d\ge \bar{d}^*=n-\deg(G)$ and $\bar{d}^*$ is the so called designed minimum distance of the code. 

The differential code $C_{\Omega}(D,G)$ is defined as 
\[
C_{\Omega}(D,G):=\{(res_{P_1}(\omega),\dots,res_{P_n}(\omega))| \omega\in\Omega(G-D)\},
\]
where $\Omega(G-D)=\{w\in\Omega(\mathcal{X})\, | \, (\omega)\ge G-D\}\cup\{0\}$. The differential code is an $[n,n-\ell(G)+\ell(G-D),d']_q$ code, where $d'\ge d^*=\deg(G)-2g+2$ and $d^*$ denotes the dual designed minimum distance. %We are really interested to these codes since it is known that $C_{\Omega}(D,G)=C_{\mathcal{L}}(D,G)^{\bot}$.

The aim of this paper is to find the minimum distance of some dual Algebraic-Geometric codes. To achieve this goal we will use a number of times the following result which is a byproduct of  \cite[Theorem 3.5]{Couv2012}.
\begin{theorem}
	\label{Th:Couvreur}
	Let $\mathcal{X}$ be a non singular curve which is complete intersection in a projective space of dimension $r$,  $D$  the divisor $D=\sum_{P\in\mathcal{X}\setminus\{P_{\infty}\}}P$, $m\geq 2$  an integer, and $d$  the minimum distance of the code $C(D,mP_{\infty})^{\bot}$. Then
	\begin{enumerate}
		\item $d=m+2$ if and only if $m+2$ points of $\mathcal{X}$ are collinear in $\mathbb{P}^{r}$;
		\item $d=2m+2$ if and only if no $m+2$ points of $\mathcal{X}$ are collinear and there exist $2m+2$ points of $\mathcal{X}$ lying on a plane conic (possibly reducible);
		\item $d = 3m$ if and only if no $m + 2$ points of $\mathcal{X}$ are collinear, no $2m + 2$ points lie on a plane conic, and there exist $3m$ points of $\mathcal{X}$ coplanar and belonging to the  intersection of a cubic curve and a curve of degree $m$ having no common irreducible component;
		\item $d \ge 3m +1$ if and only if no sub-family of the points of $\mathcal{X}$ satisfies one of the three above configurations.
	\end{enumerate}
\end{theorem}
\subsection{Affine variety codes}
We introduce now  affine variety codes, see \cite{FL1998} for further information.

Let $t\ge1$ and consider an ideal $I=\langle g_1,\dots,g_s\rangle$ of $\F_q[x_1,\dots,x_t]$,  $\{x_1^q-x_1,\dots,\, x_t^q-x_t\}\subset I$. The ideal $I$ is zero-dimensional and radical; see \cite{Sei1974}. Let $V(I)=\{P_1,\dots, P_n\}$ be the variety of $I$ and $R=\F_q[x_1,\dots,x_t]/I$.
\begin{definition}
	An affine variety code $C(I,L)$ is the image $\phi(L)$ of $L\subseteq R$, a $\mathbb{F}_q$-vector subspace of $R$ of dimension $r$, given by the isomorphism of $\F_q$-vector spaces:
	\[
	\begin{split}
	\phi:R&\longrightarrow \F_q^n\\
	f&\longmapsto (f(P_1),\dots,f(P_n)).
	\end{split}
	\]
\end{definition}
Let $L$ be generated by $b_1,\dots,\ b_r$, then the matrix
\[
H:=\left(\begin{matrix}
b_1(P_1) & b_1(P_2) & \dots & b_1(P_n)\\
b_2(P_1) & b_2(P_2) & \dots & b_2(P_n)\\
\vdots & \vdots & \vdots & \vdots\\
b_r(P_1) & b_r(P_2) & \dots & b_r(P_n)\\
\end{matrix}\right)
\]
is a generator matrix for $C(I,L)$ and a parity-check matrix for $C^{\bot}(I,L)$.
It is clear that there is a strong connection between  affine variety codes and Algebraic-Geometric codes and that, depending on the choice of $L$, they can coincide.

Since we are interested in computing the number of minimum weight codewords of particular AG codes,  next proposition will give us  a useful criterion.

\begin{proposition}[Marcolla, Pellegrini, Sala, \cite{MPS2016}]
	\label{MPS:sistema}
	Let $1\le w\le n$. Let $I=\langle g_1,\dots,g_s\rangle$ be such that $\{x_1^q-x_1,\dots,x_t^q-x_t\}\subset I$. Let $L$ be a subspace of $\F_{q^2}[x_1,\dots,x_t]/I$ of dimension $r$ generated by $\{b_1,\dots,b_r\}$. Let $J_w$ be the ideal in $\F_{q}[x_{1,1},\dots,x_{1,t},\dots,x_{w,t},z_1,\dots,z_w]$ generated by
	$$
	\begin{array}{ll}
	\sum_{i=1}^wz_ib_j(x_{i,1},\dots,x_{i,t})&\qquad \text{for } j=1,\dots,r,\\
	\\
	g_h(x_{i,1},\dots,x_{i,t})&\qquad \text{for } i=1,\dots,w\, \text{and } h=1,\dots,s,\\
	\\
	z_i^{q-1}-1& \qquad \text{for } i=1,\dots,w,\\
	\\
	\prod_{1\le l \le t}((x_{j,l}-x_{i,l})^{q-1}-1)&\qquad \text{for } 1\le j < i \le w.\\
	\end{array}$$
	
	Then any solution of $J_w$ corresponds to a codeword of $C^{\bot}(I,L)$ of weight $w$. Also, the number $\text{A}_w(C(I,L)^{\bot})$ of codewords of weight $w$ is 
	\[
	\text{A}_w(C(I,L)^{\bot})=\frac{|V(J_w)|}{w!},
	\]
	where $|V(J_w)|$ is the number of distinct solutions of $J_w$.
\end{proposition}

\section{Intersection between the GK curve and lines or conics}
In this section we study the possible intersections between a line or a plane conic and the  curve $\mathcal{GK}$ as in \eqref{Eq:GK}. In particular, we are interested in the maximum possible size of their intersections.

\begin{proposition}\label{Prop:lines}
Let $r \subset PG(3,\ell^6)$ be a line. 	Then 
$$|r \cap GK|\leq \ell^2-\ell+1.$$
Also, any $(\ell^2-\ell+1)$-secant is parallel to the $z$-axis and all the $(\ell^2-\ell+1)$ common points are not $\F_{\ell^2}$-rational.
\end{proposition}
\begin{proof}
As already mentioned, the $\mathbb{F}_{\ell^6}$-rational points of $\mathcal{GK}$ are divided into two orbits $\mathcal{O}_1=\mathcal{GK}(\mathbb{F}_{\ell^2})$ and $\mathcal{O}_2=\mathcal{GK}(\mathbb{F}_{\ell^6})\setminus \mathcal{GK}(\mathbb{F}_{\ell^2})$. %Let $P_1 \in \mathcal{GK}(\mathbb{F}_{\ell^6})$. 

Suppose that $r\cap \mathcal{GK}(\mathbb{F}_{\ell^6})$ contains at least an $\mathbb{F}_{\ell^2}$-rational point $P_1$. Without loss of generality we can assume that  $P_1=(0,0,0)$. Let $P_2=(x,y,z) \in \mathcal{GK}(\mathbb{F}_{\ell^6})\setminus \{ P_1,\ P_{\infty}\}$. This implies $x\neq 0$. An $\mathbb{F}_{\ell^6}$-rational point $P$  on the line $r$ through  $P_1$ and $P_2$ has coordinates 
$$\left(\frac{\lambda x}{1+\lambda},\frac{\lambda y}{1+\lambda},\frac{\lambda z}{1+\lambda}\right),$$
for some $\lambda \in \mathbb{F}_{\ell^6}$. If such a point belongs to  $\mathcal{GK}$ then
$$
\left(\frac{\lambda y}{1+\lambda}\right)^{\ell+1}=\left(\frac{\lambda x}{1+\lambda}\right)^\ell+\frac{\lambda x}{1+\lambda},$$
that is  
$$\lambda^{\ell+1}y^{\ell+1}=\lambda^\ell x^\ell(1+\lambda)+\lambda x(1+\lambda)^\ell.$$

The condition  $y^{\ell+1}=x^{\ell}+x$ yields $\lambda^{\ell}x^\ell+\lambda x=0.$ Therefore it is easily seen that  the number of the intersections between the line $r$ and the GK curve is exactly $\ell+1$.

By direct checking, the same happens for the line  through $P_1$ and $P_{\infty}=(1:0:0:0)$.

% and we consider the homogenized equation of the GK curve:
%\begin{equation}
%\label{Eq:GKh}
%\begin{sistema}
%Z^{\ell^2-\ell+1}T^{\ell-1}=Y^{\ell^2}-YT^{\ell^2-1}\\
%Y^{\ell+1}=X^\ell T+XT^{\ell}
%\end{sistema}
%\end{equation}
%we have that the line passing trough $P_1$ and $P_2$ is:
%\[
%s:=\begin{sistema}
%Y=0\\
%Z=0\\
%\end{sistema}
%\]
%and intersecting this line with equation (\ref{Eq:GKh}) we obtain the equation
%\[
%X^\ell T+XT^{\ell}=0.
%\]
%We obtained that in this case there are at most $\ell$ affine points in the intersection between the line $P_1P_2$ and $\mathcal{X}$.

Suppose now that $r\cap GK(\mathbb{F}_{\ell^6})$ contains no points of $\mathcal{O}_1$. Let $P_1=(x_1,y_1,z_1), P_2=(x_2,y_2,z_2) \in \mathcal{O}_2$ two points of $r$. An $\mathbb{F}_{\ell^6}$-rational point $P$ of $r$ is 
$$P=\left(\frac{x_1+\lambda x_2}{1+\lambda},\frac{y_1+\lambda y_2}{1+\lambda},\frac{z_1+\lambda z_2}{1+\lambda}\right)$$
for some $\lambda \in \mathbb{F}_{\ell^6}$.
If  $P\in GK$ then by the second equation in \eqref{Eq:GK}
\[
\left(\frac{y_1+\lambda y_2}{1+\lambda}\right)^{\ell+1}=\left(\frac{x_1+\lambda x_2}{1+\lambda}\right)^\ell+\frac{x_1+\lambda x_2}{1+\lambda}.
\]
Recalling that $y_1^{\ell+1} =x_1^\ell+x_1$ and $y_2^{\ell+1} =x_2^\ell+x_2$, we obtain 
%\[
%(y_1+\lambda y_2)^\ell(y_1+\lambda y_2)=(x_1+\lambda x_2)^\ell(1+\lambda)+(x_1+\lambda x_2)(1+\lambda)^\ell
%\]
%\[
%(y_1^\ell+\lambda^\ell y_2^\ell)(y_1+\lambda y_2)=(x_1^\ell+\lambda^\ell x_2^\ell)(1+\lambda)+(x_1+\lambda x_2)(1+\lambda^\ell)
%\]
%\[
%y_1^{\ell+1}+\lambda^\ell y_1 y_2^\ell+ \lambda y_1^\ell y_2+\lambda^{\ell+1} y_2^{\ell+1}=x_1^\ell+\lambda^\ell x_2^\ell+\lambda x_1^\ell+\lambda^{\ell+1}x_2^\ell+x_1+\lambda x_2+ \lambda^\ell x_1+\lambda^{\ell+1} x_2
%\]
\[
\lambda^\ell (x_1+ x_2^\ell -y_1y_2^\ell)+\lambda (x_1^\ell +x_2-y_1^\ell y_2)=0.
\]
If $(x_1+ x_2^\ell -y_1y_2^\ell)\neq 0$ or $(x_1^\ell +x_2-y_1^\ell y_2)\neq 0$  then $|r\cap \mathcal{GK}(\mathbb{F}_{\ell^6})|\leq \ell+1$. On the other hand, if $x_1+ x_2^\ell -y_1y_2^\ell=x_1^\ell +x_2-y_1^\ell y_2=0$ then 
\[
(x_1+x_1^\ell) +(x_2+x_2^\ell) -y_1y_2^\ell-y_1^\ell y_2=0, \qquad y_1^{\ell+1}+y_2^{\ell+1}-y_1y_2^\ell-y_1^\ell y_2=(y_1-y_2)^{\ell+1}=0,
\]
that is $y_1=y_2$. 
Finally, from $x_1+ x_2^\ell -y_1y_2^\ell=0$, we get $x_1=x_2$. This means that if  $|r\cap \mathcal{GK}(\mathbb{F}_{\ell^6})|>\ell+1$ then $r$ has equation $X=x_1$, $Y=y_1$, with $x_1^\ell+x_1=y_1^{\ell+1}$. Clearly $y_1\notin \mathbb{F}_{\ell^2}$ otherwise $P_1$ and $P_2$ belong to $\mathcal{O}_1$. By direct checking, the line $r$ has exactly $\ell^2-\ell+1$ points in common with  the curve $\mathcal{GK}$. 

\end{proof}

\begin{proposition}\label{prop:NumberOfLines}
	The total number of $(\ell^2-\ell+1)$-secants of the $\mathcal{GK}$ is $(\ell+1)(\ell^5-\ell^3)$.
\end{proposition}
\begin{proof}
Recall that $|\mathcal{O}_2|=\ell^8-\ell^6+\ell^5-\ell^3$. Also, each point in $\mathcal{O}_2$ lies on exactly one $(\ell^2-\ell+1)$-secant $r : X=x_1, Y=y_1$ such that $(r\cap \mathcal{GK}(\mathbb{F}_{\ell^6})) \subset \mathcal{O}_2$. Therefore the number of such lines is
	
	\[
	\frac{(\ell^8-\ell^6+\ell^5-\ell^3)}{(\ell^2-\ell+1)}=(\ell+1)(\ell^5-\ell^3).
	\]
\end{proof}

%Now, taking the first equation that defines the GK curve we have that:
%\[
%\left(\frac{z_1+\lambda z_2}{1+\lambda}\right)^{\ell^2-\ell+1}=y_1^{\ell^2}-y_1
%\]
%Now, since the RHS is an element of $\F_{\ell^6}$ and $GCD(\ell^6,\ell^2-\ell+1)=\ell^2-\ell+1$ we have that when $y_1^{\ell^2}-y_1\ne0$ this equation has exactly $\ell^2-\ell+1$ different solutions.

%We know that the $\F_{\ell^2}$ orbits lies on the plane $Z=0$ and it has $\ell^3+1$ points, so the number of $\F_{\ell^6}\setminus\F_{\ell^2}$-rational points is $\ell^8-\ell^6+\ell^5-\ell^3$.
%From what we said earlier, if a line is $\ell^2-l+1$-secant then it doesn't have any $\F_{\ell^2}$ rational point.

%\begin{lemma}[\cite{BMZ2017}, Lemma 3.1]
	%\label{BMZ}
	%There exist exactly $\ell^5$-$\ell^3$ planes $\pi_a : X = a$, $a\in\F_q^6$, containing $\ell^3+1$ distinct
	%$\F_{\ell^6}$-rational points of $\mathcal{X}$. Their affine points give rise to a partition of $\mathcal{X}(\F_{\ell^6})\setminus\mathcal{X}(\F_{\ell^2})$.
%\end{lemma}

%In each plane we can consider $\ell+1$ different $\ell^2-l+1$-secants. From what we said we have that the number of the $\ell^2-l+1$-secants is $(\ell+1)(\ell^5-\ell^3)$.

%We are now interested to the study of the maximal number of intersections between the GK curve and the plane conics in order to know the minimum distance of certain codes.
\begin{proposition}\label{Prop:conics}
Let $\mathcal{C}$ be a plane conic in $PG(3,\ell^6)$. Then the size $|\mathcal{C}\cap \mathcal{GK}(\mathbb{F}_{\ell^6})|$ is at most   

$$\quad 
\left\{
\begin{array}{ll}
2(\ell^2-\ell+1)& \textrm{ if } \mathcal{C} \textrm{ reducible},\\
2(\ell+1) & \textrm{ if } \mathcal{C} \textrm{ absolutely irreducible}.\\
\end{array}
\right.
$$	
\end{proposition}
\begin{proof}
Let $\mathcal{C}$ be contained in the plane defined by $G(X,Y,Z)=\alpha X+\beta Y+\gamma Z+\delta=0$. Suppose that $\mathcal{C}$ is absolutely irreducible. 

Suppose $\gamma\neq 0$. The points $P=(x,y,z)$ in $\mathcal{C}\cap \mathcal{GK}(\mathbb{F}_{\ell^6})$ satisfy
	\[
	\begin{sistema}
	z^{\ell^2-\ell+1}=y^{\ell^2}-y\\
	y^{\ell+1}=x^\ell+x\\
	ax^2+by^2+cxy+dx+ey+f=0\\
	G(x,y,z)=0,
	\end{sistema}
	\]
	where $a,b,c,d,e,f,g\in\F_{\ell^6}$. By B\'ezout's Theorem the number of pairs $(x,y)$ satisfying $y^{\ell+1}=x^\ell+x$ and $ax^2+by^2+cxy+dx+ey+f=0$ is at most $2(\ell+1)$. Clearly, for each such pair $(x,y)$ there exists a unique $z$ satisfying $G(x,y,z)=0$. Therefore $|\mathcal{C}\cap \mathcal{GK}(\mathbb{F}_{\ell^6})|\leq 2(\ell+1)$.

Suppose now $\gamma=0$ and $\beta\neq 0$. The points $P=(x,y,z)$ in $\mathcal{C}\cap \mathcal{GK}(\mathbb{F}_{\ell^6})$ satisfy
	\[
	\begin{sistema}
	z^{\ell^2-\ell+1}=y^{\ell^2}-y\\
	y^{\ell+1}=x^\ell+x\\
	ax^2+bz^2+cxz+dx+ez+f=0\\
	G(x,y)=0,
	\end{sistema}
	\]
	where $a,b,c,d,e,f,g\in\F_{\ell^6}$. As above, there are at most $\ell+1$ pairs $(x,y)$ such that  $y^{\ell+1}=x^\ell+x$ and $G(x,y)=0$. Clearly, for each such pair $(x,y)$ there exist at most $2$ values $z$ such that $ax^2+bz^2+cxz+dx+ez+f=0$, since the $\mathcal{C}$ is absolutely irreducible. Therefore $|\mathcal{C}\cap \mathcal{GK}(\mathbb{F}_{\ell^6})|\leq 2(\ell+1)$.
	
The case $\gamma=0$ and $\alpha\neq 0$ is similar and omitted.

If the conic $\mathcal{C}$ splits into two lines, then by Proposition  \ref{Prop:lines} it is clear that $|\mathcal{C}\cap \mathcal{GK}(\mathbb{F}_{\ell^6})|$ is at most  $2(\ell^2-\ell+1)$. Note that if the two lines are $(\ell^2-\ell+1)$-secants then their common point is $(0,1,0,0)\notin \mathcal{GK}$.
\end{proof}

The previous result can be generalized to a plane curve of degree $\alpha\le \ell$.

\begin{proposition}
Let $\mathcal{X}$ be a curve of degree $\alpha\le \ell$ in $PG(3,\ell^6)$. Then the size $|\mathcal{X}\cap \mathcal{GK}(\mathbb{F}_{\ell^6})|$ is at most  
$$\left\{
\begin{array}{ll}
\alpha(\ell^2-\ell+1)& \textrm{ if } \mathcal{X} \textrm{ reducible},\\
\alpha(\ell+1) & \textrm{ if } \mathcal{X} \textrm{ absolutely irreducible}.\\
\end{array}
\right.$$
\end{proposition}
%\begin{proof}
%	Let $F(X,Y,Z)=0$ be the equation of the curve and $H(X,Y,Z)=0$ the equation of the plane.  
%	\[
%	\begin{sistema}
%	Z^{\ell^2-\ell+1}=Y^{\ell^2}-Y\\
%	Y^{\ell+1}=X^\ell+X\\
%	F(X,Y,Z)=0\\
%	G(X,Y,Z)=0
%	\end{sistema}
%	\]
%	In particular we are interested to this three equations:
%	\[
%	\begin{sistema}
%	Y^{\ell+1}=X^\ell+X\\
%	G(X,Y,Z)=0\\
%	F(X,Y,Z)=0
%	\end{sistema}
%	\]
%	Since the first and the second equation give place to at most $\ell+1$ solutions and from the third one we can get the missing variable in at most $\alpha$ ways for each partial solution we have that the theorem follows.
%	Notice that $G=G(Y,Z)$ or $G=G(X,Z)$ the proof still holds.
%\end{proof}

We conclude this section with the following proposition. 
\begin{proposition}\label{Prop:cubics}
There exist $3(\ell^2-\ell+1)$ coplanar points contained in  $\mathcal{GK}(\mathbb{F}_{\ell^6})$ lying on the intersection between a cubic curve and a curve $\mathcal{Y}$ of degree $\ell^2-\ell+1$.
\end{proposition}
\begin{proof}
Let $\overline{y}\in \mathbb{F}_{\ell^6}\setminus \mathbb{F}_{\ell^2}$. Consider three lines $r_i$ of equations $X=x_i$, $Y=\overline{y}$, $i=1,2,3$, with $x_i^\ell+x_i=\overline{y}^{\ell+1}$. Such three lines are coplanar and $(\ell^2-\ell+1)$-secants; see also Proposition \ref{Prop:lines}. 

Let $\mathcal{X}$ be the plane cubic consisting of the union of $r_1$, $r_2$, and $r_3$. Clearly,  $|\mathcal{X} \cap \mathcal{GK}(\mathbb{F}_{\ell^6})|=3(\ell^2-\ell+1)$.  To conclude the proof we have to show that these points lie on a plane curve of degree $m=\ell^2-\ell+1$.

It is enough to observe that the points in $\mathcal{X} \cap \mathcal{GK}(\mathbb{F}_{\ell^6})$ are
$$(x_i,\overline{y},z_j), \qquad i=1,2,3, \ j=1,\ldots,\ell^2-\ell+1,$$
with $z_j^{\ell^2-\ell+1}=\overline{y}^{\ell^2}-\overline{y}$. Such $3(\ell^2-\ell+1)$ points are contained in the $\ell^2-\ell+1$ lines $s_j$ of equations 

$$s_j := \left\{
\begin{array}{l}
Y=\overline{y}\\
Z=z_j\\
\end{array}
\right..$$

Therefore $\mathcal{X} \cap \mathcal{GK}(\mathbb{F}_{\ell^6})$ is contained also in $\mathcal{Y}=\bigcup_{j=1}^{\ell^2-\ell+1} s_j$. 

\end{proof}

\section{Minimum distance and number of minimum weight codewords of one point codes on the GK curve}

We first determine the minimum distance of the one point AG code $C(D,mP_{\infty})^{\bot}$, where $P_{\infty}=(1,0,0,0)$ and $D=\sum_{P \in \mathcal{GK}(\mathbb{F}_{\ell^6} )\setminus \{P_{\infty}\}} P$.

\begin{proposition}\label{Prop:MinimumDistance}
	The minimum distance $d$ of $C(D,mP_{\infty})^{\bot}$ is
	\begin{enumerate}
		\item $d=m+2$ when $m\le\ell^2-\ell-1$;
		\item $d=2m+2$ when $m=\ell^2-\ell$;
		\item $d=3m$ when $m=\ell^2-\ell+1$;
		\item $d\ge3m+1$ when $m>\ell^2-\ell+1$.
		%\item $d\ge d^*$ otherwise.
	\end{enumerate} 
\end{proposition}
\begin{proof}
	We apply Theorem \ref{Th:Couvreur}. 
	\begin{enumerate}
		\item By Proposition \ref{Prop:lines} there exist $m+2\leq \ell^2-\ell+1$ collinear points in $\mathcal{GK}$ and therefore the minimum distance is $d=m+2$.
		\item If $m=\ell^2-\ell$ then $m+2= \ell^2-\ell+2$ points of $\mathcal{GK}$ cannot be collinear. Since there exist $2m+2=2(\ell^2-\ell+1)$ points contained in a reducible plane conic (see Proposition \ref{Prop:conics}) the minimum distance is exactly $d=2m+2=2(\ell^2-\ell+1)$.
		\item If $m=\ell^2-\ell+1$ then no line contains $m+2$ points  and  no plane conic contains $2m+2$ points  of $\mathcal{GK}$. By Proposition \ref{Prop:cubics} there exist plane cubics with $3m$ points which are also contained in a curve of degree $m$ having no common components with the cubic. Therefore the minimum distance is $3m=3(\ell^2-\ell+1)$.
		\item If $m>\ell^2-\ell+1$, none of the previous cases applies and therefore the minimum distance is at least  $3m+1$.
		\end{enumerate}
		 
\end{proof}

\begin{corollary}\label{Cor:MinimumDistance}
Let $d^*$ be the designed Goppa minimum distance of $C(D,mP_{\infty})^{\bot}$. The minimum distance $d$ of $C(D,mP_{\infty})^{\bot}$ is

$$d \qquad 
\left\{
\begin{array}{ll}
m+2,& m\le\ell^2-\ell-1;\\
2m+2,&m=\ell^2-\ell;\\
3m, & m=\ell^2-\ell+1;\\
\ge3m+1, & \ell^2-\ell+1<m\leq\ell^2-1;\\
\ge d^*,&m\ge \ell^2.\\
\end{array}
\right.
$$
\end{corollary}
\proof
It is enough to observe that $3m+1$ is larger than the designed minimum distance $d^*= m(\ell^3+1)-\ell^5+2\ell^3-\ell^2+2$ only when 
$$3m +1\ge  m(\ell^3+1)-\ell^5+2\ell^3-\ell^2+2 \iff m\le\ell^2-2+\frac{3\ell^2-5}{\ell^3-2}\iff  m\leq \ell^2-1.$$
\endproof

\subsection{Number of minimum weight codewords}
In this subsection we determine the number of minimum weight codewords in  $C(D,mP_{\infty})^{\bot}$   in the case $\ell-1\le m\le 2(\ell-1)$.

Recall that for  the code $C(D,mP_{\infty})^{\bot}$  the designed Goppa minimum distance is

\[
d^*= \deg \left(mP_{\infty}\right)-2g(GK)+2= m(\ell^3+1)-\ell^5+2\ell^3-\ell^2+2.
\]

%On the other hand, using Theorem \ref{Th:Couveur} it is possible to determine the exact minimum distance of the code $C(D,mP_{\infty})^{\bot}$ under certain conditions. In particular the minimum distance given by Theorem \ref{Th:Couveur} is in the form $d=km+2$, where $k\in\{1,2\}$, or $d=3m$. In any case $d\ge km$, $k\in\{1,2,3\}$.
%If 
%\[
%km \ge  km(\ell^3+1)-\ell^5+2\ell^3-\ell^2+2 \iff m\le\ell^2-2+\frac{2-\ell^2}{\ell^3}
%\] 
%then the minimum distance given by Theorem \ref{Th:Couveur} is better than $d^*$. This happens when %$m<\ell^2-2$, so we can consider the case $\ell-1\le m\le \ell^2-2$
%In particular, since $\ell\ge2$, $m\in\N$ and the maximal number of intersection between the GK curve and a line is $\ell^2-\ell+1$ we can take $\ell-1\le m\le\ell^2-\ell+1< \ell^2-2$.

Consider the ideal $I=\langle Z^{\ell^2-\ell+1}-Y^{\ell^2}+Y,Y^{\ell+1}-X^\ell-X,X^{\ell^6}-X,Y^{\ell^6}-Y,Z^{\ell^6}-Z\rangle$ and let $R=\F_{\ell^6}[x,y,z]/I$. Also, let  
\[
\mathcal{B}_{\ell,m}=\left\{X^iY^jZ^k+I\ |\ i \in [0,\dots, \ell-1],\ j \in [0,\dots, \ell^2-\ell],\ k \in [0,\dots, m]\right\}
\]
and  $L=\langle \mathcal{B}_{\ell,m} \rangle\subseteq R$. By \cite{FG2010}, $\mathcal{B}_{\ell,m}$ is a basis for the Riemann-Roch space $\mathcal{L}(mP_{\infty})$.

%\begin{remark}
%	In this case we have that $C(D,mP_{\infty})=C(I,L)$ and then $C^{\bot}(D,mP_{\infty})=C^{\bot}(I,L)$\textcolor{red}{A che serve questo remark?}\textcolor{blue}{è piuttosto inutile, era per specificare che il codice AG è pure quell' affine variety, quindi possiamo utilizzare entrambe le cose per quello che ci fa comodo}
%\end{remark}

To count  the exact number of the minimum weight codewords of $C^{\bot}(D,mP_{\infty})$ we use Proposition \ref{MPS:sistema}. Let $w\geq d\left(C(D,mP_{\infty})^{\bot}\right)$. Using the same notations, we consider the  ideal $J_w$ given by 
\[
\begin{split}
	J_w=\bigg\langle &\left\{\sum_{i=1}^{w}u_i X _i^rY_i^sZ_i^t\right\}_{ X ^rY^sZ^t\in\mathcal{B}_{\ell,m}} , \left\{Z_i^{\ell^2-\ell+1}-Y_i^{\ell^2}+Y_i \right\}_{i=1,\dots, w} , \left\{ Y_i^{\ell+1}- X _i^\ell- X _i \right\}_{i=1,\dots, w} , \\
	&\left\{  X _i^{\ell^6-1}-1 \right\}_{i=1,\dots, w} , \left\{ Y_i^{\ell^6-1}-1 \right\}_{i=1,\dots, w} , 
    \left\{ Z_i^{\ell^6-1}-1 \right\}_{i=1,\dots, w} , \\
    &\left\{ (( X _i- X _j)^{\ell^6-1}-1)((Y_i-Y_j)^{\ell^6-1}-1)((Z_i-Z_j)^{\ell^6-1}-1) \right\}_{1\le i < j \le w} \bigg\rangle .
    \end{split}
\]

A point in $V(J_{w})$ is a $4w$-tuple 
$$(\bar{x}_1,\ldots,\bar{x}_w,\bar{y}_1,\ldots,\bar{y}_w,\bar{z}_1,\ldots,\bar{z}_w,\bar{u}_1,\ldots,\bar{u}_w)\in \mathbb{F}_{\ell^6}^{4w}$$
which corresponds to a set of $w$ points $(\bar{x}_i,\bar{y}_i,\bar{z}_i)$, $i=1,\ldots,w$, in $\mathcal{GK}(\mathbb{F}_{\ell^6})$.

\begin{theorem}\label{Th:Conta1}
	Let $\ell-1\le m\le 2(\ell-1)$. The number of minimum weight codewords in $C(D,mP_{\infty})^{\bot}$ is
	\[
	A_{d}(C(D,mP_{\infty})^{\bot})=(\ell+1)(\ell^5-\ell^3)(\ell^6-1)\binom{\ell^2-\ell+1}{d}.
	\]
\end{theorem}
\begin{proof}

By Proposition \ref{MPS:sistema}, we have to count the number $d$-uples 
$$(\bar{x}_1,\ldots,\bar{x}_d,\bar{y}_1,\ldots,\bar{y}_d,\bar{z}_1,\ldots,\bar{z}_d,\bar{u}_1,\ldots,\bar{u}_d)\in \mathbb{F}_{\ell^6}^{4d}$$
which differ in the first $3d$ coordinates, and such that $\bar{z}_i^{\ell^2-\ell+1}=\bar{y}_i^{\ell^2}-\bar{y}_i$, $\bar{y}_i^{\ell+1}=\bar{x}_i^\ell+\bar{x}_i$,  and 
\begin{equation}\label{Sistema1}
	\begin{sistema}
		\bar{u}_1+\dots+\bar{u}_d=0\\
		\bar{x}_1\bar{u}_1+\dots+\bar{x}_d\bar{u}_d=0\\
		\bar{y}_1\bar{u}_1+\dots+\bar{y}_d\bar{u}_d=0\\
		\bar{z}_1\bar{u}_1+\dots+\bar{z}_d\bar{u}_d=0\\
		\vdots\\
		\bar{x}_1^{\ell-1}\bar{y}_1^{\ell^2-\ell}{z}_1^{d-2}\bar{u}_1+\dots+{x}_d^{\ell-1}\bar{y}_d^{\ell^2-\ell}{z}_d^{d-2}\bar{u}_d=0.\\
	\end{sistema}
\end{equation}
To each tuple $(\bar{x}_1,\ldots,\bar{x}_d,\bar{y}_1,\ldots,\bar{y}_d,\bar{z}_1,\ldots,\bar{z}_d,\bar{u}_1,\ldots,\bar{u}_d)$ we can associate $d$ points $(\bar{x}_i,\bar{y}_i,\bar{z}_i)$, $i=1,\ldots,d$, in $\mathcal{GK}(\mathbb{F}_{\ell^6})$. Suppose that the number of different values $\bar{y}_i$ is $\alpha\leq d\leq \ell^2-\ell$. Without loss of generality let  $\bar{y}_{1},\ldots,\bar{y}_{\alpha}$ be pairwise distinct. 

Suppose $\alpha>1$. Let $I_{i}=\{\bar{y}_j : \bar{y}_j  = \bar{y}_i\}$, for $i=1,\ldots,\alpha$. We may  suppose $|I_1|\leq |I_2|\leq \ldots \leq |I_{\alpha}|$ and let  $\beta=|I_1|$. Note that since $d\leq 2(\ell-1)$ then $\beta\leq d/2\leq \ell-1$. 
System \eqref{Sistema1}  contains the equations 
\[
	\begin{array}{lll}
		\bar{y}_1^r\bar{u}_1+\dots+\bar{y}_d^r\bar{u}_d&=&0,\\
		\bar{x}_1\bar{y}_1^r\bar{u}_1+\dots+\bar{x}_d\bar{y}_1^r\bar{u}_d&=&0,\\
		\vdots\\
		\bar{x}_1^{\ell-1}\bar{y}_1^{r}\bar{u}_1+\dots+\bar{x}_d^{\ell-1}\bar{y}_1^{r}\bar{u}_d&=&0,\\
	\end{array}
\]
for $r=0,\ldots,\alpha-1$. 
Let us define for $i=1,\ldots,\alpha$, $U_i= \sum_{j \ : \ y_j=y_i} u_j$ and $X_i^{r,s}= \sum_{j \ : \ y_j=y_i} \bar{x}_j^r\bar{z}_j^su_j$, $r=0,\ldots,\ell-1$, $s=0,\ldots,d-2$.  The above set of equations can be written as 
\[
	\begin{sistema}
		U_{1}+\cdots+U_{\alpha}=0\\
		\bar{y}_1U_{1}+\cdots+\bar{y}_\alpha U_{\alpha}=0\\
		\bar{y}_1^2U_{1}+\cdots+\bar{y}_\alpha^2U_{\alpha}=0\\
		\vdots\\
		\bar{y}_1^{\alpha-1} U_{1}+\cdots+\bar{y}_\alpha^{\alpha-1} U_{\alpha}=0\\
	\end{sistema},
	\qquad 
	\begin{sistema}
		X_{1}^{r,s}+\cdots+X_{\alpha}^{r,s}=0\\
		\bar{y}_1X_{1}^{r,s}+\cdots+\bar{y}_\alpha X_{\alpha}^{r,s}=0\\
		\bar{y}_1^2X_{1}^{r,s}+\cdots+\bar{y}_\alpha^2 X_{\alpha}^{r,s}=0\\
		\vdots\\
		\bar{y}_1^{\alpha-1} X_{1}^{r,s}+\cdots+\bar{y}_\alpha^{\alpha-1} X_{\alpha}^{r,s}=0\\
	\end{sistema}.
\]
Each of the previous systems is a Vandermonde system with the same coefficients. Since $\bar{y}_1,\ldots,\bar{y}_{\alpha}$ are pairwise distinct the unique solutions are $U_1=\cdots=U_{\alpha}=X_{1}^{r,s} =\cdots=X_{\alpha}^{r,s}=0$. 

Among $A := \{\bar{x}_i : \bar{y}_i=\bar{y}_1\}$ the number of distinct elements is at most $\gamma\leq \beta\leq \ell-1$. Suppose $\gamma>1$. Let $A=\{x_{i_1},\ldots, x_{i_\beta}\}$. Consider the systems $U_1=X_1^{1,0}=\cdots =X_1^{\gamma-1,0}=0$, \ldots, $U_1=X_1^{1,d-2}=\cdots =X_1^{\gamma-1,d-2}=0$.

Let $J_j= \{k \ : \ \bar{x}_{i_k} \in A,  \ \bar{x}_{i_k}=\bar{x}_{i_j}\}$  and $V_j^{r}= \sum_{k \in J_j} \bar{z}_{i_k}^r u_{i_k}$, $j=1,\ldots,\gamma$, $r=0,\ldots,d-2$. We may suppose that $x_{i_1},\ldots,x_{i_\gamma}$ are pairwise distinct. The previous systems can be written as 
\[
	\begin{sistema}
		V_1^r+\cdots+V_{\gamma}^r=0\\
		\sum_{j=1}^{\gamma}   \bar{x}_{i_j}V_j^r=0\\
		\sum_{j=1}^{\gamma}   \bar{x}_{i_j}^1V_j^r=0\\
		\vdots\\
		\sum_{j=1}^{\gamma}   \bar{x}_{i_j}^{\gamma-1}V_j^{r}=0\\
	\end{sistema}.
	\]
Since $x_{i_1},\ldots,x_{i_\gamma}$ are pairwise distinct the above system has as unique solutions 
$V_1^r=\cdots=V_\gamma^r=0$, $r=0,\ldots,d-2$. Since $(\bar{x}_{i_1},\bar{y}_{i_1})=\cdots =(\bar{x}_{i_\gamma},\bar{y}_{i_\gamma})$, all the values $\bar{z}_{i_1},\ldots,\bar{z}_{i_\gamma}$ must be pairwise distinct. Therefore, from 
\[
	\begin{sistema}
		u_{i_1}+\cdots+u_{i_\gamma}=0\\
		\bar{z}_{i_1}u_{i_1}+\cdots+\bar{z}_{i_\gamma}u_{i_\gamma}=0\\
		\bar{z}_{i_1}^2u_{i_1}+\cdots+\bar{z}_{i_\gamma}^2u_{i_\gamma}=0\\
		\vdots\\
		\bar{z}_{i_1}^{d-2}u_{i_1}+\cdots+\bar{z}_{i_\gamma}^{d-2}u_{i_\gamma}=0,\\
	\end{sistema}
\]
we get that the unique solution is $u_{i_1}=\cdots=u_{i_\gamma}=0$, a contradiction.

This shows that $\alpha=1$, that is $\bar{y}_1=\cdots=\bar{y}_{d}$. Using a similar argument we can prove that $\bar{x}_1=\cdots=\bar{x}_{d}$ and therefore the values $\bar{z}_i$, $i=1,\ldots,d$, are pairwise distinct. In other words,  all the points $(\bar{x}_i,\bar{y}_i,\bar{z}_i)$, $i=1,\ldots,d$, lie on a fixed line parallel to the $z$-axis. We conclude the proof computing the exact number of solution of System \eqref{Sistema1}. Since $\bar{x}_1=\cdots=\bar{x}_{d}$ and $\bar{y}_1=\cdots=\bar{y}_{d}$  this system reduces to

\begin{equation}\label{Sistema2}
\begin{sistema}
\bar{u}_1+\dots+\bar{u}_d=0\\
\bar{z}_1\bar{u}_1+\dots+\bar{z}_d\bar{u}_d=0\\
\vdots\\
\bar{z}_1^{d-2}\bar{u}_1+\dots+\bar{z}_d^{d-2}\bar{u}_d=0.\\
\end{sistema}
\end{equation}

By Proposition \ref{prop:NumberOfLines}, we have $(\ell+1)(\ell^5-\ell^3)$ different choices for the $\ell^2-\ell+1$ secant line $r$; we need $d$ points $P_i=(x_i,y_i,z_i)$, $i\in\{1,\dots,d\}$, on $r$ up to permutations. So the total number of $d$-tuples of points is
\[
(\ell+1)(\ell^5-\ell^3)\binom{\ell^2-\ell+1}{d}d!
\]
The matrix of System \eqref{Sistema2} is a Vandermonde matrix and the solution space has linear dimension $1$:  the number of the $u_i$'s is $|\F_{\ell^6}^*|=\ell^6-1$ and finally
\[
A_d= \frac{(\ell+1)(\ell^5-\ell^3)(\ell^6-1)\binom{\ell^2-\ell+1}{d}d!}{d!}=(\ell+1)(\ell^5-\ell^3)(\ell^6-1)\binom{\ell^2-\ell+1}{d}.
\]

\end{proof}

In the case $2(\ell-1)<m<\ell^2-\ell-1$ we can give a lower bound on the number of minimum weight codewords. If we consider  $d$ collinear points of the type $(\bar{x},\bar{y},\bar{z}_i)$, $i=1,\ldots,d$, then  System \eqref{Sistema1}  collapses to System \eqref{Sistema2} (note that the $\bar{z}_i$s must be pairwise distinct) and therefore the number of the corresponding $u_i$'s is $|\F_{\ell^6}^*|=\ell^6-1$. Using again Proposition  \ref{prop:NumberOfLines} we can prove the following.

\begin{theorem}
	Let $2(\ell-1)< m< \ell^2-\ell-1$. The number of minimum weight codewords in $C(D,mP_{\infty})^{\bot}$ is at least: 
	\[
	A_{d}(C(D,mP_{\infty})^{\bot})\ge(\ell+1)(\ell^5-\ell^3)(\ell^6-1)\binom{\ell^2-\ell+1}{d}.
	\]
\end{theorem}

\section*{Acknowledgements}
This research was partially supported by Ministry for Education, University
and Research of Italy (MIUR) (Project PRIN 2012 ``Geometrie di Galois e
strutture di incidenza'' - Prot. N. 2012XZE22K$_-$005)
 and by the Italian National Group for Algebraic and Geometric Structures
and their Applications (GNSAGA - INdAM).
The second author would like to thank his supervisor, Massimiliano Sala, for the helpful suggestions.

\begin{flushleft}
Daniele Bartoli\\
Department of Mathematics and Computer Science,\\
University of Perugia,\\
e-mail: {\sf daniele.bartoli@unipg.it}
\end{flushleft}

\begin{flushleft}
Matteo Bonini\\
Department of Mathematics,\\
University of Trento,\\
e-mail: {\sf matteo.bonini@unitn.it}
\end{flushleft}

\end{document}